\documentclass[11pt,reqno,tbtags]{amsart}
\numberwithin{equation}{section}

\newtheorem{theorem}{Theorem}[section]
\newtheorem{corollary}[theorem]{Corollary}
\theoremstyle{definition}
\newtheorem{remark}[theorem]{Remark}

\newtheorem*{ack}{Acknowledgement}

\newcommand{\refT}[1]{Theorem~\ref{#1}}
\newcommand{\refC}[1]{Corollary~\ref{#1}}

\newcommand{\refR}[1]{Remark~\ref{#1}}
\newcommand{\refS}[1]{Section~\ref{#1}}

\def \d{\mbox{\boldmath $d$}}

\begingroup
  \divide\count255 by 60
  \multiply\count255 by -60
  \advance\count255 by \time
  \ifnum \count255 < 10 \xdef\klockan{\the\count1.0\the\count255}
  \else\xdef\klockan{\the\count1.\the\count255}\fi
\endgroup

\newcommand\nopf{\qed}   % for theorem without proof
\newcommand\noqed{\renewcommand{\qed}{}} % for proof with explicit \qed
\newcommand\qedtag{\tag*{\qedsymbol}}

\newcommand\set[1]{\ensuremath{\{#1\}}}

\newcommand\bigpar[1]{\bigl(#1\bigr)}

\def\rompar(#1){\textup(#1\textup)}    % usage: \rompar(...)

\newcommand\ie{i.e.\spacefactor=1000}
\newcommand\eg{e.g.\spacefactor=1000}

\newcommand\cf{cf.\spacefactor=1000}
\newcommand{\as}{a.s.\spacefactor=1000}

\newcommand\bbR{\mathbb R}

\newcommand\bbZ{\mathbb Z}
\newcommand\Z{\mathbb Z}

\newcommand\go{\omega}

\newcommand\eps{\varepsilon}

\newcommand\cE{\mathcal E}
\newcommand\cR{{\mathcal R}}
\newcommand\cV{\mathcal V}

\newcommand\hX{{\widehat X}}
\newcommand\tX{X^\#}

\newcommand\E{\operatorname{\mathbb{E}}}

\renewcommand\P{\operatorname{\mathbb{P}}}

\newcommand\e{\vec{e}}

\newcommand\resist[2]{\cR(#1\leftrightarrow#2)}

\newcommand\grad{\operatorname{\bigtriangledown}}

\renewcommand{\=}{:=} %?

\providecommand\bmin{\wedge}
\newcommand\ox[1]{O(|x|^{#1})}
\newcommand\oxx[1]{O(|x|^{-#1})}
\newcommand\xx[1]{|x|^{#1}}
\newcommand\logx{\log|x|}
\newcommand\xlogx{|x|\log|x|}
\newcommand\oo{\ensuremath{[0,\infty)}}
\newcommand\hoo{h(\infty)}
\newcommand\db{\partial B_{r_0}}

\newcommand\cVz{\cV\setminus\set z}
\newcommand\Px{\P_x}
\newcommand\Ex{\E_x}
\newcommand\Po{\P_0}

\begin{document}
\title[Random walks with restarts]{Hitting times for random walks \\ with
 restarts}
\date{May 5, 2010}

\author{Svante Janson}
\address{Department of Mathematics, Uppsala University, PO Box 480,
S-751 06 Uppsala, Sweden}
\email{svante.janson@math.uu.se}
\urladdr{http://www.math.uu.se/\~{}svante/}
\author{Yuval Peres}
\address{Microsoft Research,
One Microsoft Way,
Redmond, WA 98052-6399, USA.}
\email{peres@microsoft.com}
\urladdr{http://research.microsoft.com/en-us/um/people/peres/}

\begin{abstract}
The time it takes a random walker in a lattice to reach
the origin from another vertex $x$, has infinite mean. If the walker
can restart the walk at $x$ at will, then the minimum expected
hitting time $\gamma(x,0)$ (minimized over restarting strategies)
is finite; it was called the ``grade'' of $x$ by
Dumitriu, Tetali and Winkler.
They showed that, in a more general setting,
 the {\em grade} (a variant of the ``Gittins index'') plays a
crucial role in control problems involving several Markov chains.
Here we establish several conjectures of Dumitriu et al
on the asymptotics of the grade in Euclidean lattices.
In particular, we show that in the planar square lattice,
$\gamma(x,0)$ is asymptotic to $2|x|^2\log|x|$ as $|x| \to \infty$. The proof hinges on
the local variance of the potential kernel $h$ being almost constant on the level sets of $h$.
We also show how the same method yields precise second order asymptotics for hitting times of a random walk
(without restarts) in a lattice disk.
\end{abstract}

\maketitle

\section{Introduction}

Consider a Markov chain $(X_n)$ on a (countable) state space $\cV$,
with transition probabilities $\bigpar{p(x,y)}_{x,y\in\cV}$.
%The initial state $X_0$ is not fixed;
We use $\Px$ and $\Ex$ to denote
probability and expectation in the chain with initial state $X_0=x$.

We assume that the chain is irreducible, \ie{} that
each state can be reached from any other state.
%(We are mainly interested in simple random walk on $\bbZ^d$, and the
%reader may choose to consider this special case only.)

Dumitriu, Tetali and Winkler \cite{golf2} defined a function
$\gamma(x,z)$ for pairs of states $x,z\in\cV$. This function is a
version of the Gittins index and is called the \emph{grade}; it can be
defined as follows \cite[Theorem 6.1]{golf2}:

Consider a player that starts at $x$ with the goal of reaching $z$ as
quickly as possible. Each time the player moves, the state changes
randomly according to the transition matrix $p$ of the Markov chain;
however, the player then has the option (if she finds the new state to
be too bad) to restart by an instantaneous jump back to $x$.
The grade $\gamma(x,z)$ then is the minimum, over all strategies for
restarting, of the expected number of moves until $z$ is reached.
(Thus, a restart is not counted as a separate move, but the moves
already performed are included in the total count. Note that a restart
always moves back to the original starting state $x$.)
For other equivalent definitions, and applications of the grade to
other games, see \cite{golf2}.

\begin{remark}
\label{R:strategy}
Once the grade is computed for all starting positions, then the optimal
strategies for the game above (with initial state $x$) can all be described as follows:
\emph{If the current state is $y$, then restart
if $\gamma(y,z)>\gamma(x,z)$,
but not if $\gamma(y,z)<\gamma(x,z)$;
if $\gamma(y,z)=\gamma(x,z)$, then
it does not matter whether we restart or not.}
\end{remark}

\begin{remark}
The setting in \cite{golf2} is more general than the one presented
here, since that paper allows for a cost of each move that may depend
on the present state, while we consider here only the
case of constant cost, so that the total cost is the total time.
%Our methods apply to the more general case too, but we have not pursued this.
\end{remark}

The purpose of this paper is to answer questions raised by
Dumitriu et al~\cite{golf2},
 on the asymptotics of the grade in $\bbZ^d$, $d\ge2$.
(The case $d=1$ is simple; as shown in \cite{golf2},
$\gamma(x,0)=|x|(|x|+1)$ for $\bbZ$.)

By translation invariance it clearly suffices to consider $z=0$.
Denote the Euclidean norm of $x$ by $|x|$.

\begin{theorem}\label{T:Z2}
For simple random walk on $\bbZ^2$,
\begin{equation*}
\gamma(x,0)=2|x|^2\log|x|+(2\gamma_e
%(2\mbox{\rm \boldmath $\gamma$}
  +3\log 2-1)|x|^2+O(|x|\log|x|),
\qquad |x|\ge2,
\end{equation*}
where $\gamma_e\=\lim_n (\sum_{j=1}^n \frac{1}{j} -\log n)$
is Euler's constant.
\end{theorem}

\begin{theorem}\label{T:Zd}
For simple random walk on $\bbZ^d$, $d\ge3$,
\begin{equation*}
\gamma(x,0)=\frac{\go_d}{p_d}|x|^d+O(|x|^{d-1}),
\end{equation*}
where $\go_d=\pi^{d/2}/\Gamma(d/2+1)$ is the volume of the unit ball
in $\bbR^d$, and $p_d$ is the escape probability of the simple random
walk in $\bbZ^d$,
\ie, the probability that the walk never returns to its
starting point.
\end{theorem}

The leading terms were conjectured by Dumitriu, Tetali and Winkler in the preprint version of \cite{golf2}.
%\marginal{I presume that the faulty constant in \cite{golf2} will be
%corrected before publication, and see no reason to mention it here}
Based on the heuristic argument
 that the lattice structure should be unimportant on large
scales, they suggested that a near-optimal strategy might be:
\emph{Restart if the current state has a larger Euclidean distance to
the target $0$ than the starting state}.
The expected hitting time for this strategy can then be
estimated using electrical network theory.

Theorems \ref{T:Z2} and~\ref{T:Zd} together with \refR{R:strategy}
imply the following corollary, which shows that the optimal restarting
strategy is indeed as outlined above, except possibly at some
border-line cases.

\begin{corollary}
\label{C:Z}
For simple random walk on $\bbZ^d$, with target $z=0$, there exists a
constant $C=C(d)$, independent of the starting position $x$, such that
every optimal strategy restarts from every position $y$ with
$|y|>|x|+C$, but never when $|y|<|x|-C$.
\nopf
\end{corollary}

For intermediate cases with $|x|-C \le |y| \le |x|+C$, we
cannot prescribe explicitly the optimal startegy;
numerical calculations indicate that the simple
heuristic strategy is not always optimal, \ie{} we cannot take $C=0$
in the corollary. (Peter Winkler, personal communication).
%\marginal{Check with Peter}

To prove the theorems above, we state and prove in \refS{S:general}
a result, \refT{T:A}, yielding bounds on the grade for general Markov
chains.
This theorem is applied to $\bbZ^d$ in Sections \ref{S:Z2}
and~\ref{S:Zd}.
(We separate the recurrent case $d=2$ from the transient case $d\ge3$
since the details are somewhat different.)
%
%Further applications of \refT{T:A} are briefly discussed in
%\refS{S:more}.
In \refS{S:Bm} we present analogous results for a continuous
version of the problem,
with the random walk replaced by Brownian motion in $\bbR^d$.
In this case we obtain exact results, analogous to the asymptotic
results in Theorems \ref{T:Z2} and~\ref{T:Zd}.
In the final section we show how our method yields precise second order asymptotics for hitting times of a random walk
(without restarts) in a lattice disk.

\section{A general estimate}\label{S:general}
We state a theorem yielding upper and lower bounds on the grade.
The theorem applies in principle to any Markov chain, but its
usefulness depends on the existence of a suitable harmonic function
for the Markov chain.
Recall that a function $h:\cV\to\bbR$ is {\bf harmonic} at $x\in\cV$ if
$\Ex h(X_1) =h(x)$, \ie{} if $\sum_y p(x,y)h(y)=h(x)$.

For a function $h:\cV\to\bbR$ and $x\in\cV$, define the {\bf local variance},
\begin{equation*}
V_h(x)\=\Ex|h(X_1)-h(x)|^2=\sum_y p(x,y)|h(y)-h(x)|^2.
\end{equation*}

\begin{theorem}
\label{T:A}
Let $z\in\cV$ and suppose that $h:\cV\to[0,\infty)$ is a non-negative
function that is harmonic on $\cVz$
and satisfies $h(z)=0$.

Suppose that $g_+,g_-$ are positive functions
defined on $[0,\sup h)$,
such that for every $x,y\in\cV$
with $p(x,y)>0$, and every real number $\xi$
between $h(x)$ and $h(y)$, the local variance satisfies
%$t\in[0,1]$,
\begin{equation}
\label{a}
%g_-\bigpar{h(x)+t(h(y)-h(x))}
g_-(\xi) \le V_h(x) \le g_+(\xi) \,.
%\le g_+\bigpar{h(x)+t(h(y)-h(x))},
\end{equation}
Then, for every $x\in\cV$,
\begin{equation}
\label{b}
\int_0^{h(x)} \frac{2s}{g_+(s)}\, \d s
\le \gamma(x,z)
\le
\int_0^{h^*(x)} \frac{2s}{g_-(s)}\,\d s \, ,
\end{equation}
where
\begin{equation*}
h^*(x)=\sup\set{h(y):p(w,y)>0 \text{ for some $w\in\cV$ with $h(w)\le h(x)$}}.
\end{equation*}
\end{theorem}

\begin{proof}
Fix a starting position $x_0\in\cV$.
To prove the {\bf lower bound} in (\ref{b}),
define a function $F=F_+:\oo\to\oo$ by
\begin{equation}
\label{c0}
F(s)\=\int_0^{s\bmin h(x_0)} \int_t^{h(x_0)} \frac2{g_+(u)}\,\d u\,\d
t .
\end{equation}
Thus $F(0)=0$, and by Fubini's theorem,
\begin{equation}
\label{c}
%F\bigpar{h(x_0)}
F(h(x_0))
=\iint_{0<t<u<h(x_0)} \frac 2{g_+(u)}\, \d t \, \d u
=\int_{0}^{h(x_0)} \frac {2u}{g_+(u)}\,\d u \, .
\end{equation}
For all $s\ge0$,
\begin{equation}
\label{c1}
0\le F(s)\le F(h(x_0)).
\end{equation}
Moreover,
\begin{equation*}%\label{c}
F'(s)=\begin{cases}
            \int_{s}^{h(x_0)} \frac {2}{g_+(u)}\, \d u ,& s\le h(x_0),\\
                0,& s\ge h(x_0),
          \end{cases}
\end{equation*}
and, a.e.,
\begin{equation}
\label{c2}
F''(s)=\begin{cases}
            - \frac {2}{g_+(u)},&s\le h(x_0),\\
                0,& s> h(x_0).
          \end{cases}
\end{equation}

Let $\hX_n$, $n=0,1,\dots$, be the process obtained by starting at
$\hX_0=x_0$, choosing successive states by running
the Markov chain and restarting
according to some non-anticipating strategy $\Lambda$.
(Formally, $\Lambda$ is a $\{0,1\}$-valued function on finite
sequences of states.)
That is, suppose that a step of the Markov chain
takes $\hX_n$ to $\tX_{n+1}$.
If $\Lambda(\hX_1,\ldots,\hX_n,\tX_{n+1})=0$, then we let
$\hX_{n+1}=\tX_{n+1}$, while if $\Lambda(\hX_1,\ldots,\hX_n,\tX_{n+1})=1$,
then we let $\hX_{n+1}=x_0$.

We claim that
$$
Y_n \=F\bigpar{h(\hX_n)}+n
$$
is a submartingale for any choice of restarting strategy $\Lambda$.

To see this, start by observing that
$$
F\bigpar{h(\hX_{n+1})}
\ge F\bigpar{h(\tX_{n+1})} \, ,
$$
since $F$ attains its maximum at $h(x_0)$.
Hence,  denoting $\hX_n=x$, we find that
\begin{equation}
\label{yx}
\begin{split}
\E(Y_{n+1}\mid \hX_1,\dots,\hX_n)
&\ge
\E\bigpar{F(h(\tX_{n+1}))+n+1\mid \hX_1,\dots,\hX_n}\\
&=\Ex F(h(X_1))+n+1.
\end{split}
\end{equation}

Denote $Z=h(X_1)-h(x)$. A Taylor expansion of $F$
(with error in integral form), followed by an application of
\eqref{c2} and \eqref{a}, yields
\begin{equation*}
\begin{split}
F\bigpar{h(X_1)}
&= F\bigpar{h(x)+Z}\\
&= F\bigpar{h(x)}+ZF'\bigpar{h(x)}+\int_0^1(1-t)F''\bigpar{h(x)+tZ}
Z^2\, \d t\\
&\ge F\bigpar{h(x)}+ZF'\bigpar{h(x)}
  -Z^2\int_0^1(1-t)\frac{2}{g_+\bigpar{h(x)+tZ}}\,\d t\\
&\ge F\bigpar{h(x)}+ZF'\bigpar{h(x)}
  -Z^2\int_0^1(1-t)\frac{2}{V_h(x)}\, \d t\\
&\ge F\bigpar{h(x)}+ZF'\bigpar{h(x)}
  -\frac{Z^2}{V_h(x)}.
\end{split}
\end{equation*}
If $ x \ne z$, then $h$ is harmonic at $x$, so $\Ex Z=0$ and
$\Ex Z^2=V_h(x)$.  Therefore, taking
the expectation in the last displayed inequality, we find that
$$
\Ex F\bigpar{h(X_1)} \ge F\bigpar{h(x)}-1 \,.
$$
This also holds, trivially, when $x=z$. Thus by \eqref{yx},
\begin{equation*}
\E(Y_{n+1}\mid \hX_1,\dots,\hX_n)
\ge F\bigpar{h(x)}+n = Y_n,
\end{equation*}
which proves that $(Y_n)$ is a submartingale.

We stop this submartingale at
\begin{equation}
\label{t}
\tau \= \inf\set{n: \hX_n =z}.
\end{equation}
Note that $Y_\tau= F\bigpar{h(z)}+\tau =\tau$.
Moreover, by \eqref{c1},
$$
\sup_{n\le \tau} |Y_n|
=\sup_{n\le \tau} Y_n
\le  F\bigpar{h(x_0)}+\tau.
$$
Hence, if $\E \tau<\infty$, the stopped submartingale is uniformly
integrable, and thus by the optional sampling theorem
$$
\E \tau = \E Y_\tau \ge \E Y_0 = F\bigpar{h(x_0)}.
$$
This is trivially true if $\E \tau=\infty$ too.

In other words, for any restarting strategy, the expected hitting time
of $z$ by $(\hX_n)$ is at least $ F\bigpar{h(x_0)}$, \ie{}
$\gamma(x_0,z)\ge F\bigpar{h(x_0)}$,
and the left hand side of \eqref{b} follows by \eqref{c}, since $x_0$
is arbitrary.

Next, we prove the {\bf upper bound} in \eqref{b}.
We denote the initial state by $x_0$, and use the simple
restarting strategy:
\emph{Restart to $x_0$ from all points $y=\tX_n$ with $h(y)>h(x_0)$.}

Denote the resulting process by $(\hX_n)$ and observe that
$$
h(\hX_n)\le h(x_0) \mbox{ \rm and }
h(\hX_{n+1}) \le h(\tX_{n+1}) \le h^*(x_0) \mbox{ \rm for all } n.
$$
Consider
\begin{equation}
\label{c*}
F^*(s)
=F^*_-(s)
:=\int_0^{s\bmin h^*(x_0)} \int_t^{h^*(x_0)}
\frac2{g_-(u)}\, \d u\, \d t \,.
\end{equation}
By an argument similar to the one above, we find that
\begin{equation*}
\E\bigpar{F^*(h(\hX_{n+1}))\mid \hX_1,\dots,\hX_n}
\le F^*\bigpar{h(\hX_n)}-1.
\end{equation*}
Denote $Y^*_n=F^*\bigpar{h(\hX_n)}+n$ and let $\tau$ be defined by \eqref{t}.
Then $(Y^*_{n\wedge \tau})_{n \ge 0}$ is a positive supermartingale,
whence by the optional sampling theorem,
\begin{equation*}
\gamma(x_0,z) \le
\E \tau = \E Y_\tau \le \E Y_0 = F^*\bigpar{h(x_0)}
\le F^*\bigpar{h^*(x_0)}.
\qedtag
\end{equation*}
\noqed
\end{proof}

\begin{remark}\label{R:A1}
The proof above suggests that a reasonable strategy is to restart from
every state $y$ with $h(y)>h(x)$, as in the second part of the proof.
For $\bbZ^2$, $d\ge2$, with $h$ as described in Sections \ref{S:Z2}
and \ref{S:Zd}, this is close (but not identical) to the strategy
based on Euclidean distance, and \refC{C:Z} shows that it is, in some
sense, close to optimal.
\end{remark}

\begin{remark}
To obtain matching upper and lower bounds from \refT{T:A},
we want $g_-\approx g_+$.
It is thus essential that we can find a harmonic function
$h$ such that $V_h(x)$ is approximately a function of $h(x)$, \ie{}
such that $V_h(x)$ is roughly constant in sets where $h(x)$ is.
\end{remark}

%\begin{remark}\label{R:A3}
%The proof above can be modified to yield information on where the
%process spends its time until hitting 0.
%Indeed, let $0\le a< b<h(x_0)$. If we modify \eqref{c0} by including
%the indicator function $\ett[a\le u\le b]$ as a factor inside the
%integral, it follows as above that the expected number of times $n$
%such that $h(\hX_n)\in[a,b]$ is at least
%$\int_a^b\frac{2u}{g_+(u)}\, \d u$ for any strategy. For the strategy in
%\refR{R:A1}, we obtain a corresponding upper bound. Hence the integrands
%$2u/g_\pm(u)$ can be interpreted as approximate occupation densities of
%$h(\hX_n)$.
%\end{remark}

\begin{remark} \label{R:A2.5}
The applications of \refT{T:A} below follow a common pattern,
here given as a heuristic guide to later precise calculations.
Suppose that $r(x)$ is a function on $\cV$ such that $h(x)\approx
\varphi(r(x))$ and $V_h(x)\approx\psi(r(x))$ for some $\varphi$ and $\psi$
with $\varphi$ increasing and differentiable. Suppose further that
$h(x)-h(y)$ is sufficiently small when $p(x,y)>0$.
We then can take $g_\pm(s)\approx\psi(\varphi ^{-1}(s))$ and obtain
\begin{equation*}
\gamma(x,z)
\approx\int_0^{\varphi(r(x))}\frac{2s}{\psi(\varphi^{-1}(s))}\, \d s
=\int_{\varphi^{-1}(0)}^{r(x)}\frac{2\varphi(t)\varphi'(t)}{\psi(t)}
\d t.
\end{equation*}
\end{remark}

%\begin{remark}
%We present here in a non-rigorous way, ignoring error estimates, an
%alternative argument leading to the same type of estimates.
%\marginal{Is this enough? Is it intelligible? Does it make sense?}
%
%Assume $g_+(s)\approx g_-(s)\approx g(s)$.
%We use the strategy in \refR{R:A1}; thus the process $\hX_n$ spends
%all its time at states with $0<h(\hX_n)\le h(x)$.
%Consider the total time spent with $h(\hX_n)$ belonging to an interval
%$[s,s+\gD s]\subset(0,h(x))$, where we assume that $\gD s$ is much
%larger than the discrete steps taken by $h(\hX_n)$.
%Since $h(\hX_n)$ is a martingale (in the region where we do not
%restart),
%if $h(\hX_n)\approx s$ at some time $n_0$, then $h(\hX_n)$ has
%probability $\approx \gD s/s$ of hitting 0 before it reaches
%$s+\gD s$; hence $h(\hX_n)$ will, on the average, traverse the
%interval $s/\gD s$ times from $s$ to $s+\gD s$ before it hits $0$.
%The interval will be traversed in the opposite direction one time
%more, so the total number of traversals is $\approx 2s/\gD s$.
%Moreover, since the variance of the increments of $h(\hX_n)$ is
%$V_h(\hX_n)\approx g(s)$ in this region, a comparison with a simple
%random walk on $\bbZ$ (or Brownian motion) shows that the expected
%time for each crossing is $\approx(\gD s)^2/g(s)$.
%Hence the tatal time spent with $h(\hX_n)\in[s,s+\gD s]$ is roughly
%$\frac{2s}{g(s)}\gD s$, as seen in a different way in \refR{R:A3}, and
%a summation over such intervals yields
%$\int_0^{h(x)}\frac{2s}{g(s)}\,ds$.\end{remark}

\section{Two dimensions: Proof of \refT{T:Z2}}\label{S:Z2}
In this section the underlying Markov chain $(X_n)$ is
simple random walk on $\bbZ^2$. We choose $h(x)=\frac\pi2a(x)$,
where
$$
a(x)\= \sum_{n=0}^\infty\Big[\Po(X_n=0)-\Po(X_n=x)\Big]
$$
is the potential kernel studied in \cite{Stohr,Spitzer,Lawler,FU}.
A complete asymptotic expansion of $a(x)$ is presented in \cite{FU,KS}; here we only quote the second order expansion
given, \eg{} in \cite{Stohr,FU} and \cite[Section 1.6]{Lawler}:
\begin{equation}
\label{b1}
h(x)=\tfrac\pi2a(x)=\log|x|+b+\oxx2,
\end{equation}
where
$b=\gamma_e+\tfrac32\log2$;
moreover, if $\e$ is a unit vector along one of the coordinate axis,
then
\begin{equation*}
a(x+\e)-a(x)=\e\cdot\grad\bigpar{\tfrac2\pi\logx} + \oxx2
\end{equation*}
and thus
\begin{equation*}
V_h(x)=\tfrac12\bigl|\grad(\logx)\bigr|^2+\oxx3
=\tfrac12\xx{-2}+\oxx3.
\end{equation*}

If $p(x,y)>0$, then $|x-y|=1$ and thus by \eqref{b1}
\begin{equation}
\label{b3}
h(y)=h(x)+\oxx1=\logx+b+\oxx1.
\end{equation}
It is now easily seen that \eqref{a} is satisfied with
\begin{equation}
  \label{g+-}
g_\pm(s)=\tfrac12e^{-2(s-b)}(1\pm C e^{-s})
\end{equation}
if $C$ is a sufficiently large constant.
For small $s$ this could make $g_-(s)\le0$, but we redefine $g_-(s)$
to be a small positive constant in these cases.
We then have
\begin{equation}
  \label{ginv}
\frac1{g_\pm(s)}=2e^{2(s-b)}\bigpar{1+O(e^{-s})}
\quad \text{as }
s \to \infty \,.
\end{equation}
Furthermore, \eqref{b3} implies that
$$h^*(x)=\logx+b+\oxx1.
$$
Hence \refT{T:A} yields, for $|x|\ge2$,
\begin{align*}
\gamma(x,0)&=\int_0^{\logx+b+\oxx1}
  4se^{2(s-b)}\bigpar{1+O(e^{-s})}\,\d s\\[1ex]
&=\bigl[2se^{2(s-b)}-e^{2(s-b)}+O(se^s+e^s)\bigr]_0^{\logx+b+\oxx1} \\[1ex]
&=2(\logx+b)\xx2-\xx2+O(\xlogx).
\qedtag
\end{align*}

\section{Transient case: Proof of \refT{T:Zd}}\label{S:Zd}

For simple random walk on $\bbZ^d$, $d\ge3$, we
employ  the Green function $G(x):=G(x,0)=\sum_{n=0}^\infty\Big[\Px(X_n=0)\Big]$.
We have
\cite[Section 1.5]{Lawler}
\begin{equation*}
G(x)=a_d\xx{2-d}+\oxx{d},
\end{equation*}
where $a_d=\frac{2}{(d-2)\go_d}$,
and
\begin{equation*}
V_G(x)=\frac1d\bigl|\grad(a_d\xx{2-d})\bigr|^2+\ox{1-2d}.
\end{equation*}

Let
\begin{equation*}
h(x)\=a_d^{-1}\bigpar{G(0)-G(x)}
=a_d^{-1}G(0)-\xx{2-d}+\oxx{d}
\end{equation*}
and write $h(\infty)=a_d^{-1}G(0)$.
Thus
\begin{equation*}
h(y)=h(\infty)-\xx{2-d}+\ox{1-d},\qquad p(x,y)>0,
\end{equation*}
and
\begin{equation*}
h^*(x)=h(\infty)-\xx{2-d}+\ox{1-d}.
\end{equation*}
Moreover,
\begin{equation*}
\begin{split}
V_h(x)&=
a_d^{-2} V_G(x)
=\frac1d\bigl|\grad(\xx{2-d})\bigr|^2+\ox{1-2d}\\
&=\frac{(d-2)^2}d\xx{2-2d}\bigpar{1+\oxx1}.
\end{split}
\end{equation*}
Hence we can take, for some large constant $C$ and with a modification
for small $s$ to keep the values positive,
$$
g_\pm(s)=
\frac{(d-2)^2}d\bigpar{h(\infty)-s}^{\frac{2d-2}{d-2}}
\bigpar{1\pm C (h(\infty)-s)^{\frac{1}{d-2}}}.
$$
Consequently, \refT{T:A} yields
\begin{equation} \label{gamma1}
\gamma(x,0)=\int_0^{\widetilde{h}(x)}
  \frac{2sd}{(d-2)^2} \bigpar{h(\infty)-s}^{\frac{2-2d}{d-2}}
  \bigpar{1+ O \bigpar{h(\infty)-s}^{\frac{1}{d-2}}}\,\d s \,,
\end{equation}
where $\widetilde{h}(x)=\hoo-\xx{2-d}+\ox{1-d}$.
(Recall that $\widetilde{h}(x)$ is $h(x)$ in the lower bound
 for $\gamma(x,0)$, and $h^*(x)$ in the upper bound.)

Next, we change variables to $u=u(s):=(h(\infty)-s)^{-1/(d-2)}$.
Observe that $u(\widetilde{h}(x))=|x|+O(1)$ and $\d s=(d-2)u^{1-d} \,
\d u$. If we denote $u_0=h(\infty)^{-1/(d-2)}$, then
\begin{align*}
\gamma(x,0)&=\int_{u_0}^{|x|+O(1)} \frac{2 \bigpar{\hoo-u^{2-d}}d}{(d-2)^2} u^{2d-2}
  \bigpar{1+ O (u^{-1})}(d-2)u^{1-d}\,\d u\\
&=2\hoo\frac{d}{d-2}\int_{u_0}^{|x|+O(1)}
  \bigpar{ u^{d-1}+ O(u^{d-2})} \,\d u\\
&=\frac{2\hoo}{d-2}\xx{d}+\ox{d-1}.
\end{align*}
The result follows because $G(0)=1/p_d$ and thus
\begin{equation*}
\frac{2\hoo}{d-2}
=\frac{2G(0)}{(d-2)a_d}
=\frac{\go_d}{p_d}.
\qedtag
\end{equation*}

%\section{Other Markov chains}\label{S:more}
%\refT{T:A} applies to other Markov chains too, again taking $h$ to be the
%corresponding potential kernel and using suitable $g_\pm$.
%For example, we expect results similar to Theorems \ref{T:Z2} and
%\ref{T:Zd} for random walk on other lattices, such as the triangular
%lattice, but we have not pursued this.

%Similarly, it should be possible to treat more general random walks on
%$\bbZ^d$ where $X_1-X_0$ has almost any distribution on $\bbZ^d$ with
%finite support, at least as long as $\E (X_1-X_0)=0$.
%(This would use estimates of the type given in
%\cite[Sections 12, 26, 29]{Spitzer}.)
%We do not know
%whether the method yields useful results in cases with drift.

\section{Brownian motion}\label{S:Bm}
In this section we consider a continuous analogue of the problem
studied above.
We consider Brownian motion in $\bbR^d$, starting at some given
$x\in\bbR^d$, and again we are allowed to restart at $x$ at any given
time.
Since, when $d\ge2$, the Brownian motion \as{} never will hit $0$, we
now let our target be a small ball $B_{r_0}=\set{y:|y|\le r_0}$,
where $r_0>0$ is some arbitrary fixed number. (For $d=1$ we could take
$r_ 0=0$ too.) The grade then is defined as in the discrete case, by
taking the infimum of the expected hitting time over all restarting
strategies.

Let
\begin{equation}
\label{bh}
h(x)\=
\begin{cases}
\xx{}-r_0,& d=1,\\[.2ex]
\log \bigl( \xx{}/r_0 \bigr),& d=2,\\[.3ex]
r_0^{2-d}-\xx{2-d},&d\ge3.
\end{cases}
\end{equation}
Then $h$ is harmonic and positive in the complement of $B_{r_0}$, with
$h(x)=0$ when $\xx{}=r_0$.
Moreover,
\begin{equation}
\label{bm3}
 |\grad \!\! h(x)|^2 =
\begin{cases}
1& d=1,\\
\xx{-2},& d=2,\\
(d-2)^2\xx{2-2d},&d\ge3
\end{cases}
\end{equation}
is now exactly a function of $h(x)$, say $g(h(x))$.

Let the starting point be $x_0$ and denote the process, using some
non-anticipating restarting rule,
%(given by a stopping time),
by $\hX_t$. Let further
$\tau\=\inf\set{t:|\hX_t|=r_0}$.
If $F$ is defined by \eqref{c0} (with $g_+=g$), we see as in
the proof of \refT{T:A}, now using It\^o's formula instead of a Taylor
expansion, that $F\bigpar{h(\hX_t)}+t$ is a local submartingale and,
again by the optional sampling theorem, that
$\E \tau \ge F\bigpar{h(x_0)}$.
%we leave the details to the reader.
Since this holds for any restarting strategy,
$$
\gamma(x_0,\db) \ge F\bigpar{h(x_0)}.
$$
Conversely, using the strategy
\emph{restart when $h(\hX_t)\ge h(x_0)+\eps$} for some $\eps>0$,
we find that if $F^*$ is defined by \eqref{c*} with
$h^*(x_0)=h(x_0)+\eps$ and $g_-=g$, then
$\E T \le F^*\bigpar{h^*(x_0)}$.
Letting $\eps\to0$, this and the lower bound above show, together with
\eqref{c}, that the grade is given by
\begin{equation}
\label{bm1}
\gamma(x_0,\db) = F\bigpar{h(x_0)} =\int_0^{h(x_0)}\frac{2s}{g(s)}\,\d
s.
\end{equation}

\begin{remark}
We see that the optimal strategy is to restart whenever the current
position is more distant from the origin than the starting point $x_0$,
which is the intuitively obvious strategy.
Some care has to be taken interpreting this, however, since this
\as{} entails infinitely many restarts in any interval $(0,\delta)$.
The resulting process can be obtained by taking a limit as in the proof
of \eqref{bm1} above, or by utilizing a reflected
Brownian motion (for the radial part).
\end{remark}

%To evaluate the integral in \eqref{bm1}, we write
Write $h(x)=\varphi(\xx{})$
and $|\grad h(x)|^2=\psi(x)$, so that $g(s)=\psi(\varphi^{-1}(s))$,
We obtain, \cf{} \refR{R:A2.5}, that
\begin{equation*}
\gamma(x,\db)
=\int_0^{\varphi(\xx{})}\frac{2s}{\psi(\varphi^{-1}(s))}\, \d s
=\int_{r_0}^{\xx{}}\frac{2\varphi(r)\varphi'(r)}{\psi(r)}\d r.
\end{equation*}
Taking $\varphi$ and $\psi$ from \eqref{bh} and \eqref{bm3}, we easily
evaluate this integral and deduce the following result.
%Note that we here obtain exact formulas.

\begin{theorem}
\label{refT{T:Rd}}
For Brownian motion in $\bbR^d$, if $|x|\ge r_0>0$,
\begin{equation*}
\gamma(x,\db)=
\begin{cases}
(\xx{}-r_0)^2,&d=1,\\
\xx2\logx-\xx2(\tfrac12+\log r_0)+\tfrac12 r_0^2,&d=2,\\
\frac{2r_0^{2-d}}{d(d-2)}\xx d -\frac1{d-2}\xx2+\frac1d r_0^2,&d\ge3.
\end{cases}
\end{equation*}
\vskip-18pt %-1.5\baselineskip
\nopf
\end{theorem}

It is instructive to compare these exact results for $\bbR^d$ and the
asymptotic results for $\bbZ^d$ in Theorems \ref{T:Z2} and~\ref{T:Zd}.
Note first that the time scales differ by a factor $d$, since in the
simple random walk, each coordinate of a step has variance $1/d$.
With this adjustment we see that we obtain the same leading term for
$\bbZ^d$ and $\bbR^d$ when $d\le2$; in this case, the choice of $r_0$
affects only lower order terms.
When $d\ge3$, however, we obtain the same $\xx d$ rate, but the
constant for Brownian motion depends on the choice of $r_0$, and there
is no reasonable way to obtain the right constant for $\bbZ^d$ from
the continuous limit.
This reflects the fact that the constant for $\bbZ^d$ involves the
escape probability $p_d$, which depends on the local lattice structure
near $0$ that is lost in the continuous limit.

\section{Hitting times for random walk in a disk}\label{S:disk}

The method above can also be used to estimate expected hitting times in
reversible Markov chains. For simplicity, we consider only simple random
walk on a graph $(\cV,\cE)$ with vertex set $\cV$.
We thus assume $p(x,y)=1/\deg_{\cV}(x)$ when $x\sim y$ (and 0 otherwise), where
$\deg_{\cV}(x)\=\set{y\in \cV : y\sim x}$.

\begin{theorem}
\label{T:No}
Let $z$, $h$, $g_+$ and $g_-$ be as in \refT{T:A},
for simple random walk on a graph $(\cV,\cE)$,
and let $D$ be a finite
connected subset of\/ $\cV$ with $z\in D$.
Define
\begin{align*}
  \partial D&:= \{ x\in D: x\sim y \text{ for some } y \notin D \} ,
\\
\partial ^2 D&:= \partial D \cup \{ x\in D: x\sim y   \text{ for some } y \in \partial D \} ,
\\
h_1 &:= \min \{ h(x): x\in \partial ^2 D \} ,
\\
 B&:=
\{ x \in D:
x \sim y  \text{ for some }  y \text{ with } h(y)\ge h_1\} .
\end{align*}

 Let $\{X_n\}_{n=0}^{\infty}$ be a
simple random walk on $D$. Let $\tau :=\min \{ n:X_n=z\}$.
Then, for any $X_0=x_0 \in D$,
\begin{equation}\label{xb}
  \int_0^{h_1} \frac{2(u \wedge h(x_0))}{g_+(u)}\, \d u
\le \E_{x_0} \tau
\le
\int_0^{h_1} \frac{2(u \wedge h(x_0))}{g_-(u)}\, \d u+\Delta ,
\end{equation}
where $ \Delta \= \E_{x_0} \# \{ n\le \tau :X_n \in B \}$.
\end{theorem}

Note that $h$ is harmonic on all of $\cV$, while $X_n$ is defined on $D$
with transition probabilities $p_D(x,y)\=1/\deg_D(x)$ when $x\sim y$ and
$x,y\in D$.

The error term $\Delta$ can be estimated in several ways. One of them is to bound $\tau=\tau_z$ by the time $\tau_*$ that it takes the random walk to visit $z$ and return to $x_0$.
Then (see, e.g., Lemma 10.5 and Proposition 10.6 in \cite{LPW})
\begin{equation}
\label{Delta}
\Delta \le \E_{x_0} \# \{ n \le \tau_* : X_n \in B \}
= \frac {\mu (B)}{\mu (D)} \E_{x_0}(\tau_*)
=\mu(B)\resist{x_0}z,
\end{equation}
where $\mu(B)= \sum_{x \in B} \deg_D (x)$ and $\resist {x_0}z$ is the
resistance
between $x_0$ and $z$ in $D$, regarded as an electrical network.
%The commute time identity (see, e.g. Prop.\ 10.6 in \cite{LPW}) gives a
%resistance formula for $\E_{x_0}(\tau_*)$.

\begin{proof}

We define, in analogy with \eqref{c0} and \eqref{c*},
\begin{equation}
  \label{F+-}
F_{\pm}(s):=\int_0^{s \wedge h_1} \int_t^{h_1} \frac{2}{g_{\pm}(u)}
 \d u \d t = \int_0^{h_1}\frac{2(s \wedge u)}{g_{\pm}(u)} \d u.
\end{equation}
We thus integrate only up to $h_1$, and we may redefine $g_+(u)=\infty$ for
$u > h_1$. The right hand inequality in \eqref{a} then still holds for all
$x\in \cV$, and we obtain from $x\in D \backslash \partial D$, exactly as
in \refS{S:general},
\begin{equation}
  \label{xa}
\E (F_+ (h(X_{n+1})) | X_n=x) = \E_x F_+ (h(X_1)) \ge F_+(h(X))-1.
\end{equation}
On the other hand, if $X_n =x \in \partial D$, then $X_{n+1} \in \partial^2
D$, and thus
$h(X_n)$, $h(X_{n+1})\ge h_1$ and $F_+(h(X_{n+1}))=F_+(h(X_n))$, so
\eqref{xa}
holds in this case too. Consequently, $Y_n := F_+(h(X_n))+n$ is a
submartingale, and as in \refS{S:general}
$$\E_{x_0}\tau=\E_{x_0} Y_{\tau} \ge \E_{x_0} Y_0 = F_+ (h(x_0)),$$
which is the left hand inequality of \eqref{xb}.

For an upper bound, we assume that $x\neq z$.
The argument in \refS{S:general}  works for $x\in D \backslash B$, and we obtain
\begin{equation}
  \label{xc}
\E_x F_-(h(X_1)) \le F_-(h(x))-1 ,
\qquad x\in D\backslash B.
\end{equation}
For $x\in B\backslash \partial D$, the same argument yields only, using $F_-^{''}\le 0$,
\begin{equation}
  \label{xd}
\E_x F_- (h(X_1)) \le F_-(h(x)),
\qquad x\in B\backslash \partial D.
\end{equation}
Finally, if $x\in \partial D$, then $h(x),h(X_1)\ge h_1$ for every $X_1 \sim x$, and
\begin{equation}
  \label{xe}
F_-(h(X_1))=F_-(h(x)), \qquad x\in \partial D .
\end{equation}

We define $N_n :=\# \{ k<n: X_k \in B\}$ and find from
\eqref{xc}--\eqref{xe} that  if
$Y_n^{*}:=F_-(h(X_n))+n-N_n$,
then $Y^*_{n\wedge\tau}$, $n\ge 0$, is a supermartingale and thus
$$\E_{x_0}\tau-\E_{x_0} N_{\tau}=\E_{x_0} Y_{\tau}^*\le Y_0^*=F_-(h(x_0)),$$
which completes the proof of \eqref{xb}.
\end{proof}

\subsection{Application}

Take $\cV=\bbZ^2$ with edges between vertices at distance 1.
Consider simple random walk on    the disk
 $D=\{ x\in \Z^2: |x|\le R\}$. Let $z=0 \in D$ and
take $h$ and $g_\pm$ as in \eqref{b1} and \eqref{g+-}.
Then by \eqref{ginv} and \eqref{F+-},
\begin{equation*}
\begin{split}
F_{\pm}(x_0)
&= \int_0^{h_1}\frac{2(u \wedge h(x_0))}{g_{\pm}(u)}\d u\\
&= \int_0^{h_1}4(u \wedge h(x_0))(e^{2(u-b)} \pm O(e^u))\d u\\
&=\bigl[2(u \wedge h(x_0)) e^{2(u-b)}\bigr]_0^{h_1}
 - \int_0 ^{h(x_0)\wedge h_1} 2e^{2(u-b)}
 +O\Bigl(\int_0^{h_1}h(x_0)e^u \d u\Bigr)\\
&=2(h(x_0) \wedge h_1) e^{2(h_1-b)} - e^{2(h(x_0)\wedge h_1-b)}
  +1 + O(h(x_0)e^{h_1}),
\end{split}
\end{equation*}
where $b=\gamma_e+\tfrac32\log2$.
By \eqref{b1},
\begin{align*}
  h_1&=\log (R+O(1))+b+O(R^{-2})=\log R +b+O(R^{-1}),
\\
h(x_0)&=\log |x_0| +b + O(|x_0|^{-2}).
\end{align*}
Thus
\begin{equation*}
\begin{split}
F_{\pm}(x_0)
&= 2\bigpar{h(x_0)+O(R^{-1})}e^{2\log R +O(R^{-1})}
 - e^{2 \log |x_0| +O(|x_0|^{-2}+R^{-1})} +1 +O(R\log |x_0|)\\
&=2R^2h(x_0)+O(R\log R)-|x_0|^2.
\end{split}
\end{equation*}
Further, it is easily seen that $x\in B$ implies $|x|=R-O(1)$, and thus
$\mu(B)=O(R)$. Also, it is easy to see that for $x_0$ in $D$ we have
$$
\resist{x_0}{0}=O(\log R) \, .
$$
(This follows, e.g., from the method of random paths~\cite{per} by picking a uniform point $u$ on the chord bisecting the segment $x_0z$, and
considering the lattice path closest to the union of the segments $x_0u$ and $uz$.)
 Thus \eqref{Delta} yields
$$
\Delta=O(\E_{x_0}\tau_*/R)=O(R \cdot \resist{x_0}{0})=O(R\log R).
$$
Consequently, \refT{T:No} yields for the hitting time $\tau$ of the origin, that
$$
\E_{x_0}\tau=2R^2h(x_0)+O(R\log R)-|x_0|^2.
$$
For $|x_0|\ge R^{1/2}$ say, we can write this as
$$
\E_{x_0}\tau=2R^2 \log|x_0| +2bR^2-|x_0|^2+O(R\log R).$$

\begin{ack}
This research was mostly performed during the 3rd IES workshop in Harsa and
at Institut Mittag-Leffler, Djursholm, Sweden, August 2001.
We thank Peter Winkler and other participants for
helpful discussions.
The work was completed during a visit by SJ to Microsoft, Redmond, USA, May
2010.
\end{ack}

\end{document}